\documentclass{amsart}

\usepackage{amssymb}

\newtheorem{theorem}{Theorem}[section]
\newtheorem{proposition}[theorem]{Proposition}
\newtheorem{lemma}[theorem]{Lemma}

\newtheorem{question}[theorem]{Question}

\theoremstyle{definition}
\newtheorem{definition}[theorem]{Definition}

\theoremstyle{remark}

\numberwithin{equation}{section}

\newcommand{\St}{\mathcal{S}}

\begin{document}

\title{The classification of some GK-trisections}

\author{Trent Schirmer}
\address{Department of Mathematics, Oklahoma State University, 
Stillwater, OK 74078}
\email{trent.schirmer@okstate.edu}
\urladdr{www.trentschirmer.com} 

\begin{abstract}
We classify a large class of ``unbalanced'' $4$-manifold GK-trisections, which are a slight generalization of $4$-manifold trisections defined by Gay and Kirby in \cite{Gay-Kirby}. 
\end{abstract}

\maketitle

In the following, we classify a large class of ``unbalanced'' $4$-manifold GK-trisections, which are a slight generalization of $4$-manifold trisections defined by Gay and Kirby in \cite{Gay-Kirby}.  The proof itself only occupies about five pages.  It employs a single technical lemma about Heegaard diagrams together with some famous results on Dehn surgery due to Gabai \cite{Gabai} and Gordon-Luecke \cite{Gordon-Luecke}.  

In Section 1 we give an overview of GK-trisections which emphasizes their parallelism with Heegaard splittings, and motivates the remainder of the paper.  In Section 2 we prove the above mentioned technical lemma about Heegaard diagrams.  In Section 3 we present our classification result, together with some examples of GK-trisections which we believe to be ``non-standard.''

\section{GK-Trisections and Heegaard splittings}

The following brief introduction to GK-trisections emphasizes their similarity to Heegaard splittings.  For a more detailed introduction to GK trisections, with many examples included for illustration, we recommend Gay and Kirby's original paper \cite{Gay-Kirby}.  Any book about $3$-manifolds will have a discussion of Heegaard splittings; we recommend \cite{Scharlemann} for a more sophisticated introduction.  

A manifold $Y$ is said to be {\em properly embedded} in $X$ if it is transverse to $\partial X$ and $Y\cap \partial X=\partial Y$.  A {\em proper isotopy} of $Y$ in $X$ is a homotopy of $Y$ in $X$ through proper embeddings. $N(Y,X)$ denotes a closed regular neighborhood of $Y$ in $X$, $E(Y,X)=\overline{X-N(Y,X)}$ is the {\em exterior} of $Y$ in $X$, and $Fr(Y,X)=N(Y,X)\cap E(Y,X)$ denotes the {\em frontier} of $Y$ in $X$.  We shall often drop the ambient space $X$ from this notation when no confusion can arise.  A {\em genus $g$ $n$-handlebody} is $\natural^gS^1\times D^{n-1}$, unless $g=0$, in which case it is $D^n$.  In other words, it is an $n$-dimensional ball with $g$ oriented $1$-handles attached to it.

\begin{definition}
A genus-$g$ {\em Heegaard splitting} of a closed $3$-manifold $M$ is a triple $(H_1,H_2,\Sigma)$ satisfying the following conditions:

\begin{itemize}
\item $H_1\cup H_2=M.$
\item $H_i$ is a genus $g$ $3$-handlebody, $i=1,2.$
\item $\Sigma=H_1\cap H_2=\partial H_1=\partial H_2$.

\end{itemize}

In this case, $\Sigma$ is said to be a {\em Heegaard surface} of $M$.

\end{definition}

\begin{theorem}

Every closed, orientable, connected $3$-manifold admits a Heegaard splitting.

\end{theorem}

One analogue of Heegaard splittings for $4$-manifolds, discovered by Gay and Kirby via their theory of Morse $2$-functions, is that of a {\em GK trisection}.

\begin{definition}

A {\em $(g,k_1,k_2,k_3)$-trisection} of a closed, smooth $4$-manifold $X$ is a quadruple $(V_1,V_2,V_3,\Sigma)$ satisfying the following conditions:

\begin{itemize}

\item $V_1\cup V_2\cup V_3=X$

\item $V_i$ is a genus $k_i$ $4$-handlebody, $1\leq i\leq 3$.

\item $H_{ij}=V_i\cap V_j$ is a genus $g$ $3$-handlebody whenever $i\neq j$.

\item $\Sigma=V_1\cap V_2\cap V_3$ is a closed orientable genus $g$ surface.

\item $\partial V_i=H_{ij}\cup H_{ik}$, whenever $i, j, k$ are distinct.

\end{itemize}

\end{definition}

Although it did not matter in Definition 1.1 and Theorem 1.2, it is essential to state here that in Definition 1.3 we are in the smooth category.  Each of the $4$-dimensional handlebodies $V_i$ has a corner along $\Sigma$, and the central $3$-dimensional object $H_{12}\cup H_{13}\cup H_{23}$ is nicely embedded in the following sense: Each $H_{ij}$ is smoothly embedded, as is $\Sigma$.  Moreover, each $H_{ij}$ intersects $N(\Sigma)\cong \Sigma\times D^2$ in a subset of the form $\Sigma\times R$, where $R\subset D^2$ is ray emanating from the origin of $D^2$.  With these assumptions in place, we have:

\begin{theorem} \cite{Gay-Kirby}
Every closed, orientable, connected, smooth $4$-manifold admits a $(g,k,k,k)$-trisection for some pair of integers $g,k\geq 0$.
\end{theorem}

In Definition 1.3 above, $(H_{ij}, H_{ik},\Sigma)$ is a Heegaard splitting of $\partial V_i\cong \#^{k_i} S^1\times S^2$ whenever $i,j,$ and $k$ are distinct.  It is to these triples of handlebodies that the traditional theory of Heegaard splittings can be applied with some success, as is shown in the work of Meier-Zupan \cite{Meier-Zupan} and the present work.

\begin{definition}

A {\em handlebody tripod} is a union $H_1\cup H_2\cup H_3$ of $3$-dimensional handlebodies such that, for all $i\neq j$, $(H_i,H_j,H_i\cap H_j)$ forms a Heegaard splitting of $S^3$ or $\#^kS^1\times S^2$ for some $k$.

\end{definition}

The handlebodies $H_{12}\cup H_{13}\cup H_{23}$ of Definition 1.3 form a handlebody tripod, which we call the {\em characteristic tripod} of the trisection.

\begin{proposition} \cite{Gay-Kirby} For any handlebody tripod $H_1\cup H_2\cup H_3$ there is a unique (up to diffeomorphism) $4$-manifold $X$ for which $H_1\cup H_2\cup H_3$ forms the characteristic tripod of a GK-trisection.

\end{proposition}

It follows that handlebody tripods can be used to distinguish smooth manifolds from one another.  In particular, if a topological $4$-manifold $X$ admits multiple smooth structures, each smooth structure will have a distinct class of handlebody tripods associated with it.  Naturally, the case when $X=S^4$ is of special interest, so before moving forward we'll make a few motivational observations about it.

The {\em topological} Poincar\'{e} conjecture is known to be true in dimension $4$ by Freedman's work \cite{Freedman}, and the homotopy type of a simply connected $4$-manifold is determined by its intersection form.  Therefore any closed, simply connected $4$-manifold $X$ with Euler characteristic $\chi(X)=2$ is homeomorphic to $S^4$.  Moreover, it is an elementary exercise to show that if $X$ admits a $(g,k_1,k_2,k_3)$-trisection, then $\chi (X)=2+g-k_1-k_2-k_3$. This yields the following:

\begin{proposition}

If $X$ admits a $(k_1+k_2+k_3,k_1,k_2,k_3)$-trisection and $\pi_1(X)=0$, then $X$ is homeomorphic to $S^4$.

\end{proposition}

Since a $4$-manifold $X$ is obtained from any of its characteristic tripods $Y$ by attaching $3$- and $4$-handles, the map $\pi_1(Y)\rightarrow \pi_1(X)$ induced by inclusion is an isomorphism.  Therefore, the class of simply connected $(k_1+k_2+k_3,k_1,k_2,k_3)$ handlebody tripods is precisely the class of characteristic tripods for smooth manifolds $X$ which are homeomorphic to $S^4$.  Finding a nice way of characterizing the (proper?) subclass of such tripods which impose the standard smooth structure on $S^4$ would thus be an important step to resolving the last open case of the generalized Poincar\'{e} conjecture.

Gay and Kirby provide us with our first tool in understanding the class of handlebody tripods which define a given smooth manifold--it is the analogue of the Reidemeister-Singer theorem for Heegaard splittings.  We require the some new definitions to state these theorems.

\begin{definition}

Let $\mathcal{H}=(H_1,H_2,\Sigma)$ be a Heegaard splitting of $M$ and let $\alpha$ be a boundary parallel arc properly embedded in $H_2$.  Then if $H_1'=H_1\cup N(\alpha)$, $H_2'=E(\alpha, H_2)$ and $\Sigma'=H_1'\cap H_2'$, the triple $\mathcal{H}'=(H_1',H_2',\Sigma')$ is also a Heegaard splitting of $M$, called the {\em stabilization} of $\mathcal{H}$.  Conversely, we say that $\mathcal{H}$ is a {\em destabilization} of $\mathcal{H}'$.

\end{definition}

The stabilization of a Heegaard splitting is unique up to isotopy.  That is, if $(H_1',H_2',\Sigma')$ and $(H_1'',H_2'',\Sigma'')$ are both stabilizations of the same Heegaard splitting of $M$, then there is an isotopy of $M$ taking $\Sigma'$ to $\Sigma''$ and $H_i'$ to $H_i''$ for $i=1,2$ (in general we say that two Heegaard splittings are isotopic in this case).  Destabilizations do not always exist and when they do they are not always unique up to isotopy.  The following is a classical result of Reidemeister and Singer.

\begin{theorem} \cite{Reidemeister},\cite{Singer}
Suppose $\mathcal{H}_1$ and $\mathcal{H}_2$ are Heegaard splittings of $M$.  Then there are integers $n$ and $m$ such that the result of stabilizing $\mathcal{H}_1$ $n$ times is isotopic to the result of stabilizing $\mathcal{H}_2$ $m$ times.
\end{theorem}

Using the lingo, one says that any two Heegaard splittings of a $3$-manifold are ``stably isotopic.'' Using a very similar notion of stabilization for GK-trisections, Gay and Kirby proved that any two trisections are stably isotopic.

\begin{definition}

Given a trisection $\mathcal{T}=(V_1,V_2,V_3)$ of $X$, a {\em $1$-stabilization} of $\mathcal{T}$ is a trisection $\mathcal{T}'=(V_1\cup N(\alpha),E(\alpha, V_2),E(\alpha, V_3))$, where $\alpha\subset V_2\cap V_3$ is a properly embedded, $\partial$-parallel arc. An {\em $i$-stabilization} is defined similarly for $i=2,3$. Conversely $\mathcal{T}$ is said to be an $i$-{\em destabilization} of $\mathcal{T}'$.

\end{definition}

If $\mathcal{T}$ is a $(g,k_1,k_2,k_3)$-trisection, then the $1$-stabilization of $\mathcal{T}$ is a $(g+1,k_1+1,k_2,k_3)$ trisection (and similarly for $i=2,3$).  In Definition 1.8, it did not matter whether the arc $\alpha$ was properly embedded in $H_1$ or $H_2$, as the resulting Heegaard splittings will be isotopic. With trisections, the result of an $i$-stabilization is generally not isotopic to a $j$-stabilization when $i\neq j$ (we say that two trisections $(V_1,V_2,V_3,\Sigma)$ and $(V_1',V_2',V_3',\Sigma')$ of $X$ are isotopic if there is an isotopy of $X$ taking $V_i$ to $V_i'$, for $i=1,2,3$).  However, any two $i$-stabilizations of a trisection are isotopic.  Moreover, $i$- and $j$-stabilizations commute, that is, the result of performing an $i$-stabilization followed by a $j$-stabilization is isotopic to the result of performing a $j$-stabilization followed by an $i$-stabilization.

\begin{theorem} \cite{Gay-Kirby} Any two GK-trisections of a smooth $4$-manifold $X$ are stably isotopic.

\end{theorem}

Aside from Theorem 1.11, we know almost nothing about the set of GK-trisections associated with a smooth manifold $X$.  Even the case of $S^4$ (equipped with its standard smooth structure) is not well understood.  It admits the trivial $(0,0,0,0)$-trisection $\mathcal{S}_0$, which is unique up to isotopy, as well as the trisections $\mathcal{S}^{k_1,k_2,k_3}$ obtained by performing $k_i$ $i$-stabilizations of $\mathcal{S}_0$ for $1\leq i\leq 3$ (also unique up to isotopy).  The following question is open:

\begin{question}

Does $S^4$, with its standard smooth structure, admit any GK-trisections other than $\mathcal{S}_0$ and trisections of the form $\St^{k_1,k_2,k_3}$?

\end{question}

The answer to Question 1.12 has a significance beyond $S^4$.  To see why, note that the connect sum operation extends naturally to GK-trisections using the formula: $$(V_1,V_2,V_3,\Sigma)\#(V_1',V_2',V_3',\Sigma')=(V_1\natural V_1',V_2\natural V_2',V_3\natural V_3',\Sigma\#\Sigma').$$  For example, in this terminology the $1$-stabilization of a trisection $\mathcal{T}$ is $\mathcal{T}\#\mathcal{S}^{1,0,0}$.  More generally, given any trisections $\mathcal{T}$ and $\mathcal{S}$ of $X$ and $S^4$, respectively, we can obtain another trisection $\mathcal{T}\#\mathcal{S}$ of $X$.  Therefore any complexity in the set of trisections for $S^4$ will be inherited by all closed smooth $4$-manifolds.

Waldhausen proved that the correlate of Question 1.12 has a positive answer for Heegaard splittings of $S^3$. In fact, we will make heavy use of the following theorem.

\begin{theorem} \cite{Waldhausen} Suppose $M=\#^{k}S^1\times S^2$ (where we define $\#^0S^1\times S^2=S^3$).  Then for all $g\geq k$, there is a unique genus $g$ Heegaard splitting of $M$ up to isotopy.

\end{theorem}

This theorem, though incredibly useful, highlights the main difficulty in understanding trisections of $4$-manifolds.  Theorem 1.13 tells us that if $H_1\cup H_2\cup H_3$ is a $(g,k_1,k_2,k_3)$-handlebody tripod, then each pair $(H_i,H_j,H_i\cap H_j)$ of Heegaard splittings is essentially unique and, moreover, quite simple as far as Heegaard splittings go.  The complexity of a trisection is therefore not contained in the way that each pair of handlebodies is glued together, but is irreducibly tied to the simultaneous combination of all three handlebodies.

However, it is at first difficult to see how complicated trisections can be constructed at all, given the restrictive manner in which each pair of handlebodies must be glued together.  Indeed, Meier and Zupan have shown that there are (up to diffeomorphism) only a handful of distinct $(g,k,k,k)$ trisections for $0\leq k\leq g\leq 2$, and they are all ``standard'' \cite{Meier-Zupan}.  This fact is by no means obvious--their proof is a combinatorial {\em tour de force}.

Our main focus in this paper shall be on the classification of $(g,k_1,k_2,0)$-trisections when $k_1\geq g-1$.  If $X$ admits such a trisection, then it is simply connected, so $\chi(X)\geq 2$ and thus $k_1+k_2\leq g$ (more generally, $k_1+k_2+k_3\leq g+2\min_i\{k_i\}$).  It is easy to show that all $(g,g,k_1,k_2)$-trisections satisfy $k_1=k_2$ and are very simple.  In particular, the only $(g,g,0,0)$-trisection is $\St^{g,0,0}$ (see Proposition 3.4 below).  Most of our work goes into the classification of $(g,g-1,k,0)$-trisections given in Theorem 3.6 below, which can be restated here as follows.\\

\noindent{\bf Theorem 3.6.}  {\em Up to diffeomorphism, the only $(g,g-1,1,0)$-trisection is $\St^{g-1,1,0}$, and the only $(g,g-1,0,0)$-trisection is $\mathcal{T}\#\St^{g-1,0,0}$, where $\mathcal{T}$ is the standard genus $1$ trisection of $\mathbb{CP}^2$.}\\

Theorem 3.6 gives a positive answer to Question 1.12 for the case of $(g,g-1,1,0)$-trisections.  At the end of Section 3 we describe a class of GK-trisections which appear to be somewhat exotic.  If this is true, then the simplicity of Theorem 3.6 does not carry over to the case of $(g,k_1,k_2,0)$-trisections when $k_1<g-1$.  Moreover, if our examples are exotic, then a positive answer to Question 1.12 for all $(g,k,g-k,0)$-trisections implies that the Poincar\'{e} conjecture is false.

\section{A handle sliding lemma}

In this section, the word ``handlebody'' always refers to a $3$-handlebody.  An ``essential'' simple closed curve on a surface is one which does not bound a disk on that surface.  The purpose of this section is to prove Lemma 2.8.

\begin{definition}

Let $M$ be a $3$-manifold and $\alpha=\alpha_1\cup\cdots\cup \alpha_n$ be a disjoint union of essential simple closed curves embedded in $\partial M$. Then $M(\alpha)$ shall denote the $3$-manifold obtained by attaching $2$-handles to $M$ along the curves $\alpha$.

\end{definition}

\begin{definition}

Suppose $H$ is a genus $g$ handlebody and $\alpha=\alpha_1\cup\cdots \cup \alpha_n$ is a disjoint union of essential simple closed curves embedded in $\partial H$.  Let $M$ be the manifold obtained from $H(\alpha)$ by capping off any spherical boundary components with $3$-handles.  Then $\alpha$ is said to be a {\em Heegaard diagram} of $M$ on $H$.  $\alpha$ is said to be {\em dualized} if there exists a union of disks $D=D_1\cup\cdots \cup D_n$ properly embedded in $H$ such that $|D_i\cap \alpha_j|=\delta_{ij}$ for all $1\leq i,j\leq n$, and in this case $D$ (as well as $\partial D$) is said to {\em dualize} $\alpha$.

\end{definition}

\begin{definition}
Suppose $\alpha_1$ and $\alpha_2$ are a disjoint pair of essential, non-isotopic simple closed curves embedded on an orientable surface $\Sigma$, and that $\beta$ is an arc embedded in $\Sigma$ which meets each $\alpha_i$ exactly once on an endpoint.  Then $\partial N(\alpha_1\cup\beta\cup\alpha_2)$ has exactly one component which is not isotopic to $\alpha_1$ or $\alpha_2$, call it $\alpha_1'$.  $\alpha_1'\cup \alpha_2$ (or just $\alpha_1'$) is said to be the result of a {\em slide} of $\alpha_1$ over $\alpha_2$ along $\beta$.
\end{definition}

If $\Sigma$ is a boundary component of a $3$-manifold $M$, and $\alpha_1$ and $\alpha_2$ bound disks $D_1$ and $D_2$ properly embedded in $M$, then $\alpha_1'$ also bounds a disk $D_1'$ in $M$.

\begin{definition}

Suppose $\alpha=\alpha_1\cup\cdots \cup \alpha_n$ is an indexed disjoint union of essential, pairwise non-isotopic simple closed curves embedded in $\Sigma$.  Then an {\em $(i,j)$-slide} of $\alpha$ is the indexed union of curves $\alpha_1\cup\cdots\cup\alpha_i'\cup \alpha_j\cup\cdots \cup \alpha_n$ obtained from by sliding $\alpha_i$ over $\alpha_j$ along a curve whose interior is disjoint from $\alpha$ 

\end{definition}

To be clear, in an $(i,j)$-slide we remove $\alpha_i$ from $\alpha$ and replace it with the new curve $\alpha_i'$ which we label with the same index.  The indexing of these curve collections is essential to the following definition.

\begin{definition}

Suppose $\gamma=\gamma_1\cup\cdots \cup \gamma_n\subset \Sigma$ is obtained from $\alpha=\alpha_1\cup\cdots \cup \alpha_n$ via a sequence of $(i,j)$ slides such that $1\leq j\leq k$ at every stage of the sequence.  Then $\gamma$ is said to be {\em slide equivalent} to $\alpha$ rel $\alpha_1\cup\cdots \cup \alpha_k$.

\end{definition}

A disjoint union $D$ of disks properly embedded in a handlebody $H$ is said to be {\em complete} if $E(D,H)$ is a ball.  The following lemma is well known, see \cite{Johannson}.  

\begin{lemma}

Any two complete collections of disks in a handlebody are slide equivalent.

\end{lemma}

In our proof of the main lemma of this section, we shall frequently perform the following operation.

\begin{definition}

Let $D$ be a disk properly embedded in a $3$-manifold $M$ (so $\partial D\subset \partial M$), and let $\delta$ be a disk (non-properly) embedded in $M$ such that $\delta\cap D\subset\partial \delta$ is an arc properly embedded in $D$ and $\delta\cap \partial M\subset \partial \delta$ is the complementary arc $\overline{\partial \delta-D}$. Then $Fr(D\cup\delta,M)$ consists of three disks, one of which is isotopic to $D$.  The other two components of $Fr(D\cup\delta,M)$ are said to be the result of an {\em outermost disk surgery} of $D$ along $\delta$.

Similarly, if $A$ is an annulus properly embedded in $M$ and $\delta$ is a disk embedded in $M$ such that $\delta\cap A\subset \partial \delta$ is a non-spanning arc in $A$, and $\overline{\partial \delta-A}=\delta\cap \partial M$, then $Fr(A\cup\delta,M)$ consists of two annuli and one disk.  One annulus will be isotopic to $A$, the other may not be, and is said to be the result of an outermost disk surgery on $A$ along $\delta$.  

\end{definition}

The following lemma is a translation of a special case of the main result of \cite{Scharlemann-Thompson} into the language of handle slides.  Our proof involves the notion of a Heegaard splitting of a $3$-manifold with non-empty boundary, together with an application of the extended version of Haken's lemma appearing in \cite{Casson-Gordon}.  The unfamiliar reader is referred to \cite{Scharlemann} for an account of these concepts.

\begin{lemma}

Suppose $\alpha=\alpha_1\cup\cdots\cup \alpha_g$ is a Heegaard diagram of $S^3$ on a handlebody $H$, and that $H(\alpha')$ is a handlebody, where $\alpha'=\alpha_1\cup\cdots \cup \alpha_k$.  If $\alpha-\alpha'$ is dualized in $H(\alpha')$, then $\alpha$ is slide equivalent rel $\alpha'$ to a dualized diagram on $H$.

\end{lemma}

\begin{proof}

Set $M'=H(\alpha')$.  Let $\alpha''=\alpha-\alpha'$, and let $\Delta'$ denote the union of the $2$-handles we have attached to $H$ along $\alpha'$, which appear as embedded solid tubes in $M'$. Let $D''=D_{k+1}\cup\cdots \cup D_g$ be the complete collection of compressing disks for $M'$ which dualize $\alpha''$.

After a slight isotopy of $D''$ in $M'$ we can ensure that $\partial D''$ is disjoint from $\Delta'\cap \partial M'$.  Once this is done, we may choose a small collar $J$ of $\partial M'$ disjoint from $\alpha'$, and a parameterization map $h:\partial M'\times I\rightarrow J$ such that:

\begin{itemize}

\item $h^{-1}(\Delta'\cap J)=(\Delta'\cap \partial M')\times I,$

\item $h^{-1}(D''\cap J)=(D''\cap \partial M')\times I$.

\end{itemize}  

Let $V=\overline{H-J}$, let $W=J\cup\Delta'$, and let $\Sigma=\partial V=\partial W$.  Then the triple $(V,W,\Sigma)$ forms a Heegaard splitting of $M'$ (see \cite{Scharlemann} for a definition of a Heegaard splitting of a $3$-manifold with boundary), so by a natural extension of Haken's lemma \cite{Casson-Gordon}, $D''$ can be isotoped, leaving $\partial D''$ fixed, so that $D_i\cap \Sigma$ is a single simple closed curve for all $k<i\leq g$.

For the remainder of the proof we shall identify the curves on $\partial H$ with their projection onto $\Sigma$ along the $I$-fibers of $J$ (note that $\Sigma-J$ already coincides with $\partial H-J$).  Beware, however, that after the isotopy of $D''$ in the previous paragraph, the projection of $\partial D''$ onto $\Sigma$ along the $I$-fibers of $J$ may differ from $D''\cap \Sigma$.  In fact, except in the simplest cases, the latter will intersect $\alpha$ many times, in a way that cannot be removed via isotopy on $\Sigma$.  

We get control of the intersections as follows.  Let $A''=A_{k+1}\cup \cdots \cup A_g=\alpha''\times I\subset J$, and let $B''=B_{k+1}\cup \cdots \cup B_g=D''\cap W$.  $A''$ and $B''$ are both disjoint, incompressible unions of annuli.  Let $E'$ denote the collection of disks in $W$ with boundary $\alpha'$.  In other words, $E'$ is just the union of the core disks of $2$-handles $\Delta'$.

Using a standard innermost disk argument we can eliminate all circles of intersection in $B''\cap (E'\cup A'')$.  Also, by construction $B_i\cap A_j$ meets $\partial M'$ in one point if $i=j$, and never otherwise.  Thus we may assume that, aside from a single spanning arc in $B_i\cap A_i$ for each $i$, $B''\cap (E'\cup A'')$ consists only of non-spanning arcs in $W$ which have both endpoints in $\Sigma$.  

We will show that the non-spanning arcs of $B''\cap (E'\cup A'')$ can be eliminated via a sequence of outermost disk surgeries on $E'\cup A''$, and that these disk surgeries correspond to a sequence of handleslides of $\alpha$ rel $\alpha'$ (under our identification of $\partial H$ with $\Sigma$, $\alpha=(E'\cup A'')\cap \Sigma$).

If there are any non-spanning arcs in $B''\cap (E'\cup A'')$ then there will be an ``outermost'' disk $\delta\subset B''$ such that $\delta\cap (E'\cup A'')=\overline{\partial \delta-\partial B}$ is a connected subarc of $\partial \delta$.  The complementary subarc in $\partial\delta$, which coincides with a subarc $\omega$ of $\partial B''$, will lie on $\Sigma$.  Both endpoints of the arc $\omega$ will lie on a single component of $\alpha$. 

Consider first the case that both endpoints of $\omega$ lie on a component $\alpha_j$ of $\alpha''$. Then since $\delta$ is disjoint from $E'$, and the $2$-handles $\Delta'$ are just a regular neighborhood of $E'$ in $W$, after a slight isotopy we may assume $\delta$ lies inside the collar $J$ of $\partial M'$. Thus $\delta$ must be $\partial$-parallel, and it follows that $\omega$ cobounds a disk $F$ with a subarc $a$ of $\alpha_j$ in the component $\Sigma'$ of $\partial J$ which partially coincides with $\Sigma$.  The interior of $F$ may contain some of the disk ``scars'' $\Delta\cap \Sigma'$, but it will be disjoint from $\alpha''-\alpha_j$.  Therefore, the curve $(\alpha_j-a)\cup\omega$ is obtained from $\alpha_j$ via a sequence of handle slides over $\alpha'$.  And $(\alpha_j-a)\cup \omega$ is the boundary of the annulus obtained from $A_j$ via a disk surgery along $\delta$.

If, instead, $\delta$ meets a component $E_j$ of $E'$, then essentially the same argument shows that a disk surgery of $E_j$ along $\delta$ corresponds to a slide of $\alpha_j=\partial E_j$ over $\alpha'$. However there is one additional detail to consider because, unlike the first case, we will have two distinct choices of disk surgery to consider, and the choice we make matters.

Specifically, as in the first case, we may assume that $\delta$ essentially lies in $J$ (a small subdisk of it will lie in the component $N(E_j,W)$ of $\Delta'$).  If $G$ and $G'$ denote the ``scar'' disks $N(E_j,W)\cap \Sigma'$, then $\omega$ intersects $\Sigma'$ in a single arc which meets $G$, say, and cobounds a disk $F\subset \Sigma'$ with a subarc in $\partial G$ such $\mathring{F}\cap\mathring{G}=\emptyset$.  This follows from the fact that $\delta\cap J$ must be $\partial$-parallel.

One disk that results from a surgery of $E_j$ along $\delta$ can be isotoped onto $F$ in $W$, while the other can be isotoped onto $F\cup G$.  If the other ``scar'' disk $G'$ lies in $F$, perform the former disk surgery, if $G'\cap F=\emptyset$, perform the latter surgery.  In the former case, $\partial G'$ can be slid over $\alpha'$ onto $\partial F$ in $\Sigma$, and in the latter case, $\partial G$ can be slid over $\alpha'$ onto $\partial (F\cup G)$ in $\Sigma$.  In both cases this corresponds to a sequence of slides of $\alpha_j$ over $\alpha'$ on $\Sigma$ as required.

This same argument may be repeated (for convenience, continue to call the new collection of disks and annuli $E'\cup A''$) until  all non-spanning arcs of $B''\cap(E'\cup A'')$ are eliminated.  Let $\gamma'=\partial E'$ and $\gamma''=\partial A''\cap \Sigma$ be the set of curves that result.  Then $\gamma''$ is dualized by $D''$, $\gamma'$ is disjoint from $D''$, and $\gamma$ is equivalent to $\alpha$ rel $\alpha'$.

It follows that $H(\gamma'')$ is a handlebody and, moreover, that there is a ball $K$ trivially embedded in $H(\gamma'')$ so that $K\cap \partial H(\gamma'')$ is a disk on $\partial K$, $K\cap \gamma'=\emptyset$, and $\Delta''\subset K$, where $\Delta''$ is the union of the $2$-handles attached to $H$ along $\gamma''$.  Theorem 1.13 and Lemma 2.6 then tell us that the curves $\gamma'$ can be slid over themselves in $\partial H(\gamma'')$ (along arcs disjoint from $K$) to a collection of curves (call it $\gamma'$ also) which is dualized by a complete collection of disks $D'$ of $H(\gamma'')$, and we of course may isotope $D'$ away from $K$ if necessary.  The final result is then dualized on $H$, and is slide equivalent to $\alpha$ rel $\alpha'$, as required.

\end{proof}

\section{The classification of $(g,g,k_1,k_2)$- and $(g,g-1,k,0)$-trisections}

Given a genus $g$ $3$-handlebody $H$, a {\em complete disk collection} $D=D_1\cup\cdots \cup D_g$ for $H$ is a disjoint union of properly embedded disks in $H$ such that $E(D,H)$ is a ball.  The union of curves $\partial D\subset \partial H$ is then said to be a {\em defining set of curves} for the handlebody $H$ on $\partial H$.  Here is a slight variation of Definition 2.2 for Heegaard splittings, which also extends in the obvious way to GK-trisections.

\begin{definition}

A {\em Heegaard diagram} for a Heegaard splitting $(H_\alpha,H_\beta,\Sigma)$ is a pair $(\alpha,\beta)$ where $\alpha$ and $\beta$ are defining sets of curves for $H_\alpha$ and $H_\beta$ on $\Sigma$.  Similarly, a {\em trisection diagram} for a given trisection tripod $H_\alpha\cup H_\beta\cup H_\gamma$ is a triple $(\alpha,\beta,\gamma)$ where $\alpha$, $\beta$, and $\gamma$, are defining sets of curves embedded in $\Sigma$ for $H_\alpha$, $H_\beta$, and $H_\gamma$, respectively.

\end{definition}

For the rest of this paper, we will use the following conventions for our indices:  If the tripod $H_\alpha\cup H_\beta\cup H_\gamma$ defines a $(g,k_1,k_2,k_3)$ trisection with trisection diagram $(\alpha,\beta,\gamma)$, then:

\begin{itemize}

\item $(H_\alpha,H_\beta,\Sigma)$ is a genus $g$ Heegaard splitting of $\#^{k_1}S^1\times S^2$
\item $(H_\alpha,H_\gamma,\Sigma)$ is a genus $g$ Heegaard splitting of $\#^{k_2}S^1\times S^2$
\item $(H_\beta,H_\gamma,\Sigma)$ is a genus $g$ Heegaard splitting of $\#^{k_3}S^1\times S^2$.

\end{itemize}

\begin{definition}

A Heegaard diagram $(\alpha,\beta)$ on a genus $g$ surface is said to be {\em $g,k$-standard} if the $\alpha$ and $\beta$ curves can be indexed so that:

\begin{itemize}

\item $\alpha_i=\beta_i$ for all $1\leq i\leq k$.

\item $|\alpha_i\cap \beta_j|=\delta_{ij}$ for all $i>k$.

\end{itemize}

\end{definition}

A $g,k$-standard Heegaard diagram defines the genus $g$ Heegaard splitting of $\#^kS^1\times S^2$.

\begin{definition}

A trisection diagram $(\alpha,\beta,\gamma)$ is {\em standard} if each of the Heegaard diagrams $(\alpha,\beta)$, $(\alpha,\gamma)$ and $(\beta,\gamma)$ is standard.

\end{definition}

The following proposition is an easy consequence of Theorem 1.13.

\begin{proposition}
Every $(g,g,k_1,k_2)$ trisection satisfies $k_1=k_2$ and can be represented by a standard trisection diagram.
\end{proposition}

\begin{proof}

Suppose that $(\alpha,\beta,\gamma)$ is a diagram for our $(g,g,k_1,k_2)$-trisection.  With our conventions, $(H_\alpha,H_\gamma,\Sigma)$ is the genus $g$ Heegaard splitting of $\#^{k_1}S^1\times S^2$, so in fact we may rechoose $(\alpha,\gamma)$ to be a standard genus $g$ diagram of $\#^{k_1}S^1\times S^2$.  Since the handlebodies $H_\alpha$ and $H_\beta$ form a genus $g$ Heegaard splitting of $\#^{g}S^1\times S^2$, Waldhausen's theorem implies we must have $H_\alpha=H_\beta$.  Hence every defining set of curves for $H_\alpha$ is also a defining set of curves for $H_\beta$, so we can choose $\alpha=\beta$.  With these choices $(\alpha,\beta,\gamma)$ is standard and $k_1=k_2$.

\end{proof}

It follows that the only $(g,g,0,0)$-trisection of any kind is the trisection $\St^{g,0,0}$ of $S^4$.  Likewise, the only $(g,g,k,k)$-trisection of any kind is $(\#^k\mathcal{G})\#\St^{g-k,0,0}$, where $\mathcal{G}$ is the standard $(1,1,1,1)$-trisection of $S^1\times S^3$.

\begin{lemma}

Suppose $H_\alpha\cup H_\beta\cup H_\gamma$ is a $(g,k_1,k_2,0)$-trisection tripod, and that $(\alpha,\beta)$ is a $g,k_1$-standard diagram of $(H_\alpha,H_\beta,\Sigma)$, so that $\alpha_i=\beta_i$ for all $1\leq i\leq k_1$. Then if $M=H_\gamma(\beta_1\cup\cdots \cup \beta_{k_1})$ is a handlebody and $\beta_{k_1+1}\cup\cdots \cup \beta_g\subset \partial M$ is dualized in $M$, then the tripod $H_\alpha\cup H_\beta\cup H_\gamma$ admits a diagram $(\alpha,\beta,\gamma)$ such that $(\alpha,\beta)$ and $(\beta,\gamma)$ are standard. 

\end{lemma}

\begin{proof}

By Lemma 2.8, there is a sequence of handleslides of $\beta$ rel $\beta'=\beta_1\cup\cdots\cup\beta_k$ to a dualized collection of curves on $H_\gamma$, which means the resulting set of curves will be standard with respect to some defining set of curves for $H_\gamma$.  But this sequence of slides also defines a sequence of slides of the $\alpha$ curves rel $\alpha'=\alpha_1\cup\cdots\cup\alpha_k$ which keeps $\alpha$ standard with respect to $\beta$, as follows.

Suppose $s$ is an arc connecting $\beta_i$ to $\beta_j$, $j\leq k$, which defines a slide of the $\beta$ curves.  Since $j\leq k$, $\beta_j=\alpha_j$, so one endpoint of $s$ lies on $\alpha_j$.  It is possible that $s$ does not meet $\alpha$ anywhere else, in which case there is nothing to do. Otherwise, let $p$ be the last point of $\alpha$ that $s$ meets before it reaches $\alpha_j$, and let $s'$ be the subarc of $s$ with endpoints $p$ and $s\cap \alpha_j$. If $p$ lies on $\alpha_r$, $s'$ defines a handleslide of $\alpha_r$ over $\alpha_j$.  This slide reduces the intersection of $\alpha$ with $s$ and leaves $\alpha\cap \beta$ unchanged.  Thus we may perform these slides as necessary until $\alpha$ is disjoint from $s$ (except at $\alpha_j$), and then perform the slide of $\beta_i$ over $\beta_j$ along $s$, all without changing $\beta\cap \alpha$.

\end{proof}

\begin{theorem}

All $(g,g-1,k,0)$-trisections are standard.

\end{theorem}

\begin{proof}

For the elementary homological reasons mentioned in Section 1, we must have $k=0,1$.  We consider the case $k=1$ first.

Let $H_\alpha\cup H_\beta\cup H_\gamma$ be the trisection tripod of a $(g,g-1,1,0)$-trisection with central surface $\Sigma$, and let $(\alpha,\beta)$ be a standard diagram of $(H_\alpha,H_\beta,\Sigma)$, indexed so that $\alpha_i=\beta_i$ for $1\leq i\leq g-1$.

Since $(H_\beta,H_\gamma,\Sigma)$ is a Heegaard splitting of $S^3$, $H_\gamma(\beta_1\cup\cdots \cup \beta_{g-1})$ is a knot complement $E(K)$ in $S^3$.  But $\alpha_1\cup\cdots \cup \alpha_{g-1}=\beta_1\cup\cdots \cup \beta_{g-1}$, so we also have $E(K)=H_\gamma(\alpha_1\cup\cdots \cup \alpha_{g-1})$, and $\alpha_g$ lies on the boundary torus $\partial E(K)$.  Since $(H_\alpha,H_\gamma,\Sigma)$ is a Heegaard splitting of $S^1\times S^2$, $\alpha_g$ must be the slope of a Dehn surgery on $K\subset S^3$ which yields $S^1\times S^2$.  By Gabai's theorem \cite{Gabai}, $K$ is the unknot and $\alpha_g$ bounds a disk in the solid torus $E(K)$.  Lemma 3.6 now tells us that we may perform slides so that $(\alpha,\beta)$ and $(\beta,\gamma)$ are standard.

It remains to show that $\alpha$ can be slid so that it is standard with respect to $\gamma$, while remaining standard with respect to $\beta$.  But $\alpha_i=\beta_i$ for all $1\leq i\leq g-1$, hence $|\alpha_i\cap \gamma_j|=\delta_{ij}$ for all $1\leq i\leq g-1$ and $1\leq j\leq g$.  Moreover, $\alpha_g$ is isotopic to $\gamma_g$ in $\partial H_\gamma(\alpha_1\cup\cdots \cup \alpha_{g-1})$, which means that $\alpha_g$ can be slid over the curves $\alpha_1\cup\cdots \cup \alpha_{g-1}$ in $\Sigma$ until it is isotopic to $\gamma_g$.  Since none of these slides effect $\alpha\cap \beta$, we are done.

The case $k=0$ is nearly identical.  The main difference is that we deduce that $E(K)$ is a solid torus using the Gordon-Luecke theorem \cite{Gordon-Luecke}.

\end{proof}

It seems unlikely that Theorem 3.6 can be extended to all $(g,k_1,k_2,0)$-trisections with $k_2\leq k_1<g-1$.  Consider a $k$-component link $L\subset S^3$ with the following properties:

\begin{itemize}

\item $L$ admits a $\#^kS^1\times S^2$ Dehn surgery.

\item The tunnel number $t(L)$ of $L$ satisfies $t(L)\geq k$.

\end{itemize}

The tunnel number of $L$ is the minimum number of arcs $t_1\cup\cdots \cup t_n$ that need to be properly embedded in $E(L)$ so that $E(L\cup t_1\cup\cdots \cup t_g)$ becomes a handlebody.  Such links exist, for example, \cite{Gompf-Scharlemann-Thompson} describes many of them in the case $k=2$.

Suppose that $L=L_1\cup L_2$ is such a link, where each component $L_i$ is connected.  Let $t_1\cup\cdots \cup t_n\subset E(L)$ be a system of unknotting tunnels for $L$ with $n=t(L)\geq 2$.  After some edge slides, we can assume that only one arc, say $t_n$, connects $\partial N(L_1)$ to $\partial N(L_2)$.  Let $\alpha_i$, $i=1,2$, be the pair of surgery slopes on $\partial N(L)$ which give $\#^2S^1\times S^2$, let $\beta_i$ be the meridian of $L_i$ on $\partial N(L_i)$, $i=1,2$.  Finally, let $\mu_i$ be the meridian curve of $t_i$ on $\partial E(L\cup t_1\cup\cdots\cup t_n)$, $1\leq i\leq n-1$.

With this information, we construct a $(n+1, n-1,2,0)$-trisection tripod $H_\alpha\cup H_\beta\cup H_\gamma$ with central surface $\Sigma=\partial E(L\cup t_1\cup\cdots \cup t_n)$ by setting:

\begin{itemize}

\item $H_\alpha=$ the handlebody defined by $\alpha_1\cup\alpha_2\cup \mu_1\cup\cdots \cup \mu_{n-1}\subset \Sigma$.

\item $H_\beta=$ the handlebody defined by $\beta_1\cup\beta_2\cup\mu_1\cup\cdots \cup \mu_{n-1}\subset \Sigma$.

\item $H_\gamma=E(L\cup t_1\cup\cdots \cup t_n)$.

\end{itemize}

It seems unlikely that these trisections are really just $\St^{n-1,2,0}$ in disguise.  The method of proof that underlies Theorem 3.6 breaks down badly because the result of attaching $2$-handles along the double curves $\mu_1\cup\cdots \cup \mu_{n-1}$ of the diagram $(\alpha,\beta)$ to $H_\gamma$ yields $E(L\cup t_n)$, which is not even a handlebody.

\section{Acknowledgements}

I would like to thank Dave Gay for helpful conversations and Marty Scharlemann for making me aware of \cite{Scharlemann-Thompson}.  I would also like to thank Charles Frohman for sitting down with me in Iowa City to look at Gay and Kirby's paper for the first time.

\end{document}